\documentclass[a4paper,12pt]{amsart}
\usepackage{amssymb}

\usepackage[looser]{newtxtext}
\usepackage{newtxmath}
\usepackage[top=35mm, bottom=35mm, left=30mm, right=30mm]{geometry}

\numberwithin{equation}{section}

\newtheorem{thm}{Theorem}[section]
\newtheorem{lem}[thm]{Lemma}

\newtheorem{prop}[thm]{Proposition}

\theoremstyle{definition}
\newtheorem{defn}[thm]{Definition}

\newtheorem{ques}[thm]{Question}
\newtheorem{rem}[thm]{Remark}

\newcommand{\bbr}{\mathbb{R}}
\newcommand{\bbn}{\mathbb{N}}

\newcommand{\orb}{\mathrm{Orb}}

\newcommand{\per}{\mathrm{Per}}

\newcommand{\htop}{h_{\mathrm{top}}}

\DeclareMathOperator{\Int}{int}

\usepackage{hyperref}

\begin{document}

\title{Graph maps with zero topological entropy and sequence entropy pairs}

\author[J. Li]{Jian Li}
\address[J. Li]{Department of Mathematics,
	Shantou University, Shantou 515063, Guangdong, China}
\email{lijian09@mail.ustc.edu.cn}

\author[X. Liang]{Xianjuan Liang}
\address[X. Liang]{Department of Mathematics,
	Shantou University, Shantou 515063, Guangdong, China}
\email{liangxianjuan@outlook.com}

\author[P. Oprocha]{Piotr Oprocha}
\address[P. Oprocha]{AGH University of Science and Technology, Faculty of Applied
	Mathematics, al.
	Mickiewicza 30, 30-059 Krak\'ow, Poland
	-- and --
	National Supercomputing Centre IT4Innovations, Division of the University of Ostrava,
	Institute for Research and Applications of Fuzzy Modeling,
	30. dubna 22, 70103 Ostrava,
	Czech Republic}
\email{oprocha@agh.edu.pl}


\begin{abstract}
We show that graph map with zero topological entropy
is Li-Yorke chaotic if and only if it has an NS-pair (a pair of  non-separable points containing in a same solenoidal $\omega$-limit set), and a non-diagonal pair is an NS-pair 
if and only if it is an IN-pair if and only if it is an IT-pair. This completes characterization of zero topological sequence entropy for graph maps.
\end{abstract}


\maketitle

\section{Introduction}

In the famous  work \cite{LT75} by Li and Yorke, it is shown that
if a continuous interval map  $f\colon [0,1]\to [0,1]$ has a periodic point with period $3$ then it is Li-Yorke chaotic.
From this result, well-known Sharkovskii's theorem and relations between periodic points and entropy,
it is easy to deduce that if a continuous map  $f\colon [0,1]\to [0,1]$
has positive topological entropy then it is Li-Yorke chaotic.
On the other hand, in 1986, Xiong  \cite{X86} and Sm\'\i{}tal \cite{S86} independently constructed 
interval maps with zero topological entropy which are Li-Yorke chaotic.
Furthermore, in \cite{S86} Sm\'\i{}tal  introduced the concept of non-separable points and showed that any interval map $f$ with zero topological entropy is Li-Yorke chaotic if and only if there is an infinite $\omega$-limit set containing two non-separable points, where
points $x,y$ are separable if there exist two disjoint periodic intervals $I$ and $J$ for $f$ such that $x\in  I$ and $y\in J$ and non-separable otherwise. 
Soon after \cite{S86} was published, two new characterizations of Li-Yorke chaotic interval maps were developed.
In 1989, Kuchta and Sm\'\i{}tal \cite{KS89} showed that an interval map has a  scrambled pair if and only if it is Li-Yorke chaotic, i.e. one scrambled pair indicates that there exists an uncountable scrambled set. Later, in 1991, Franzov\'a and Sm\'\i{}tal \cite{FS91} 
showed that an interval map is Li-Yorke chaotic if and only if it has positive topological sequence entropy. 
We refer the reader to the book \cite{R17} for more details on chaos theory of interval maps and even graph maps,
and to the survey \cite{LY16} for chaos theory of general topological dynamical systems.

Many results for interval maps 
have been extended to circle maps or even graph maps. While it is natural to expect that some characterizations will transfer to this case, usually many new technical problems arise.
A particular example of this kind of difficulty is the condition $J_1\subset f(J_1)$ on closed arc $J_1$ which ensures periodic point on $[0,1]$ but not on the circle.
In 1990, Kuchta \cite{K90} proved that a circle map with zero topological entropy is Li-Yorke chaotic if and only if it has an infinite $\omega$-limit set containing two non-separable points
if and only if it has a scrambled pair. 
In 2000, Hric \cite{H00} proved that a circle map is Li-Yorke chaotic if and only if it has positive
topological sequence entropy. Analogous problems for graph maps were considered in \cite{RS14} by
Ruette and Snoha, who proved that a scrambled pair for a graph map implies that the map is Li-Yorke chaotic. Finally,
Li, Oprocha, Yang and Zeng proved in \cite{LOYZ17} that any continuous map on a topological graph is Li-Yorke chaotic if and only if it has positive topological sequence entropy. 

A closer look at \cite{LOYZ17} shows that there are numerous problems when localizing sequence entropy (in terms of IN-pairs) compared to earlier characterization on interval \cite{L11}.
To solve this problem, we seek for inspiration in mentioned above results on interval maps, and
introduce a concept of non-separable points for graph maps and NS-pair which to some extent mimick dynamical behavior observed in interval maps (a pair of points in a solenoidal $\omega$-limit set with some additional properties; see Section~\ref{sec:NSpairs} for definitions). The first main result of this paper is as follows.  

\begin{thm}\label{thm:chaos-NS-pair}
Let $ f\colon G\to G$ be a graph map with zero topological entropy.
Then $f$ is Li-Yorke chaotic if and only if 
there exists an NS-pair. 	
\end{thm}

In the framework of so-called ``local entropy theory", 
lots of notions were introduced to describe specific dynamical properties, see \cite{GY09} for a recent survey.
Origins of this approach can be derived from works of Blanchard from 1990s (e.g. see \cite{Blanchard}), who introduced the notion of entropy pair.
It is shown in \cite{Blanchard} that a topological dynamical system has positive topological entropy if and only if it has a non-diagonal entropy pair. 
In \cite{HY06} Huang and Ye characterized an entropy pair by means of an interpolation property on a set of neighborhoods of the pair along a subset of $\bbn$ with positive density.
This concept was later expressed in term of independence sets \cite{KL07} by Kerr and Li, giving raise to non-diagonal IE-pairs characterizing positive topological entropy.
In \cite{KL07} Kerr and Li also introduced the concept of non-diagonal IN-pairs and non-diagonal IT-pairs, 
which characterize positive topological sequence entropy and non-tameness.

In 2011, Li \cite{L11} showed that for an interval map with zero topological entropy, a non-diagonal pair is non-separable if and only if it is an IN-pair if and only if it is an IT-pair. 
In 2017, Li, Oprocha, Yang and Zeng \cite{LOYZ17} showed that a graph map is Li-Yorke chaotic if and only if it has an IN-pair if and only if it has an IT-pair.  It was not known, however, if the sets of these pairs coincide. Motivated by the results obtained in previous years for interval maps,
the authors in \cite{LOYZ17} proposed an open question as follows. 

\begin{ques}\label{que1}
	Let $f\colon G\rightarrow G$ be a graph map with zero topological entropy. Is it true that:
	$\left\langle x, y\right\rangle $ is an IN-pair if and only if $\left\langle x, y\right\rangle $ is an IT-pair?
\end{ques}
Recently in \cite{Yini} Yang proved that this question has a positive answer in the case of circle maps by an elegant application of lifts of circle maps and reduction to the interval case.
By the tools involved, this approach cannot be generalized further to arbitrary graphs, however it gives a hope for a positive answer to the Question above.
Indeed, in this paper, we will completely answer  Question~\ref{que1} proving the following theorem:
\begin{thm}\label{thm:IN=IT=NS}
	Let $f\colon G\to G$ be a graph map with zero topological entropy and $\langle x,y\rangle\in G\times G$ with $x\neq y$.
	Then the following statements are equivalent:
	\begin{enumerate}
		\item $\langle x,y\rangle$ is an NS-pair; \label{Graph:NS}
		\item $\langle x,y\rangle$ is an IN-pair; \label{Graph:IN}
		\item $\langle x,y\rangle$ is an IT-pair.\label{Graph:IT}
	\end{enumerate}
\end{thm}

The paper is organized as follows.
In Section 2 we recall several useful facts on topological dynamics, with special emphasis on interval maps and graph maps case. 
Section~3 analyses the structure of $\omega$-limit sets in graph maps.
Section~4 is devoted to proof of Theorem~\ref{thm:chaos-NS-pair} and
Theorem~\ref{thm:IN=IT=NS} is proved in Sections~5.

\section{Preliminaries}

\subsection{Topological dynamics}
By a \emph{topological dynamical system}, we mean a pair $(X,f)$,
where $X$ is a compact metric space with a metric $d$ and $f\colon X\to X$ is a continuous map.
A point $x\in X$ is \emph{periodic} (for $f$) with period $n$ 
if $f^n(x)=x$ and $f^i(x)\neq x$ for
$1\leq i<n$. We denote by $\per(f)$ the set of all periodic points of $f$.
For a point $x\in X$, the \emph{orbit} of $x$ is defined as $\orb_f(x)=\{f^n(x)\colon x\in\bbn_0\}$,
and the \emph{$\omega$-limit set} of $x$ is  
\[
\omega_f(x)=\bigcap_{n=1}^\infty \overline{\{f^m(x)\colon m\geq n\}}.
\]
We also write 
\[
\omega(f)=\bigcup_{x\in X}\omega_f(x).
\]
It is well known that for every $k\in\bbn$ and $x\in X$ we have  
\[
\omega_{f}(x)=\bigcup_{i=0}^{k-1}\omega_{f^k}(f^i(x))
\]
and therefore $\omega(f)=\omega(f^k)$.
A pair $\langle x,y\rangle \in X\times X$ of points in $X$ is \emph{asymptotic} if 
\[
\lim_{n\to\infty} d(f^n(x),f^n(y))=0,
\]
\emph{proximal} if it satisfies
\[
\liminf_{n\to\infty} d(f^n(x),f^n(y))=0,
\]
and \emph{scrambled} if it is proximal but not asymptotic.
A dynamical system $(X,f)$ is \emph{Li-Yorke chaotic}
if there exists an uncountable subset $S$ of $X$ such that any two distinct points $x,y\in S$ form a scrambled pair.
It follows by results of \cite{RS14} that if a graph map has at least one scrambled pair then it is Li-Yorke chaotic.

Let $(X,f)$ and $(Y,g)$ be two dynamical system. 
If there exists a continuous surjective map $\pi\colon X\to Y$ such that 
$\pi(f(x))=g(\pi(x))$ for any $x\in X$, then we say that $(X,f)$ and $(Y,g)$ are \emph{semi-conjugate}, or $(Y,g)$ is a \emph{factor} of $(X,f)$,
and $\pi$ is a \emph{semi-conjugacy} or a \emph{factor map}.

The \emph{topological entropy} of $(X,f)$  and the \emph{sequence topological entropy} of $(X,f)$ along an infinite subset $A$ of $\bbn$
are denoted by $\htop(X,f)$ and $\htop^A(X,f)$ respectively. We refer the reader to the textbook \cite{W82} for basic results on entropy theory.
We say that the dynamical system $(X,f)$ is \emph{null} if $\htop^A(X,f)=0$ for any infinite subset $A$ of $\bbn$.

The enveloping semigroup $\mathcal{E}(X, f ) $ of a dynamical system $(X,f)$ is defined as the closure of 
$\{f^n\colon n\in\bbn_0\}$ in the product space $X^X$ endowed with the topology of pointwise convergence. 
A dynamical system $(X, f)$ is \emph{tame} if the cardinal number 
of its enveloping semigroup is not greater than the cardinal number of $\bbr$. The idea of tameness was introduced by K\"ohler in \cite{K95}. Here we adopt the approach of Glasner \cite{G06} (see also \cite{KL07}).

\begin{defn}
	Let $(X,f)$ be a topological dynamical system and 
	$ A_1, A_2,\cdots, A_k \subset X$.
	We call $I\subset\bbn$ 
	\emph{an independence set} of $\left\lbrace A_1, A_2,\cdots, A_k \right\rbrace$ if for any non-empty finite subset $J$ of $I$ and any $S\in \{1,2,\dotsc,k\}^J$ we have
	$\bigcap\nolimits_{i\in J}f^{-i}A_{S(i)}\neq \emptyset$.
\end{defn}

\begin{defn}
	Let $(X,f)$ be a topological dynamical system.
	A pair $\left\langle x, y\right\rangle  \in X\times X$ is called an \emph{IN-pair} (resp.\ an \emph{IT-pair}; \emph{IE-pair} ) if for any neighborhoods $U_1$ and $U_2$
	of $x$ and $y$ respectively, $\left\lbrace U_1, U_2 \right\rbrace $ has arbitrarily large finite independence set 
	(resp.\ $\left\lbrace U_1, U_2 \right\rbrace $ has an infinite  independence set; resp. independence set of positive upper density).
	Denote the set of IN-pairs and IT-pairs of $(X,f)$ by
	$IN(X,f)$ and $IT(X,f)$ respectively.
\end{defn}

It is clear from the definition that every IT-pair is also an IN-pair. The following theorem shows that the converse is not true (as long as we remember, that there are tame systems which are not null). The following two results were proved in \cite{KL07} (cf. also earlier \cite{WH} and the notion of weak scrambled pair).
For completeness, let us also recall, that existence of non-diagonal IE-pairs is equivalent to positive topological entropy and that every IE-pair is an IT-pair but not conversely (see \cite{KL07} for more details; cf. \cite{GY09}).

\begin{thm}[]\label{thm:null-IN}
Let $(X,f)$ be a topological dynamical system.
Then 
\begin{enumerate}
	\item $(X,f)$ is null if and only 
	if every IN-pair must be in the diagonal of $X\times X$;
	\item $(X,f)$ is tame if and only if every IT-pair must be in the diagonal of $X\times X$.	
\end{enumerate}	
\end{thm}

\begin{lem}[\cite{KL07}]\label{INIT-conjugate}
	Let $\pi\colon (X,f)\to (Y,g)$ be a factor map between two topological dynamical systems.     
	Then $\pi\times \pi (IN(X,f))=IN(Y,g)$ 
	and $\pi\times \pi (IT(X,f))=IT(Y,g)$.
\end{lem} 

The following result is folklore, see e.g. \cite{TYZ10} or \cite{Yini}.
\begin{lem} \label{lem:IT-IN-fk}
Let $(X,f)$ be a topological dynamical system. 
Then 
\begin{enumerate}
	\item both $IT(X, f )$ and $IN(X, f )$ are closed and 
	$f \times f$-invariant subsets of $X\times X$; 
	\item for every $k\in\bbn$, $IN(X,f)=IN(X,f^k)$ and  $IT(X,f)=IT(X,f^k)$.
\end{enumerate}
\end{lem}

\subsection{Interval maps}
Let $I\subset \bbr$ be a non-degenerate bounded closed interval and let $f\colon I\to I$ be a continuous map.
In this case, we say that $f$ is an interval map and usually use the map $f$ to represent the dynamical system $(I,f)$.

The following result first appeared in \cite{Sh66}, see also in \cite{Block96}.
\begin{thm}\label{thm:interval limit set}
	Let $f\colon I\to I$ be an interval map.
	\begin{enumerate}
		\item If $\omega_1$ and $\omega_2$ are two $\omega$-limit sets and $a \in \omega_1\cap \omega_2$ is a limit point from the left (resp., from the right) of both $\omega_1$ and $\omega_2$, then $\omega_1\cup \omega_2$ is also an $\omega$-limit set of $f$;
		\item If $\omega_1\subset \omega_2 \subset \cdots$ is a sequence of $\omega$-limit sets, then the closure of their union is also an $\omega$-limit set of $f$.
	\end{enumerate}	
\end{thm}

Let $f\colon I\to I$ be a continuous map.
We say that two points $x,y\in I$ are \emph{separable} if 
there exist two disjoint non-degenerate closed subintervals $I_1$ and $I_2$ of $I$ and $k_1,k_2\in\bbn$ such that
$I_i,f(I_i),\dotsc,f^{k_i-1}(I_i)$ are pairwise disjoint and $f^{k_i}(I_i)  =I_i$ for $i=1,2,$
 $x\in  I_1$ and $y\in I_2$. Otherwise we say that they are \emph{non-separable}.

In what follows, we will refer to the following two results on interval maps.
\begin{thm}[{\cite[Theorem 2.2]{S86}}]
Let $f\colon I\to I$ be an interval map with zero topological entropy.
Then $f$ is Li-Yorke chaotic if and only if there exists an infinite $\omega$-limit set containing two non-separable points.
\end{thm}

\begin{thm}[{\cite[Theorem 2.7]{L11}}] \label{thm:interval-IN-IT}
Let $f\colon I\to I$ be an interval map with zero topological entropy and $x,y\in I$ with $x\neq y$. Then the following are equivalent:
\begin{enumerate}
	\item $x,y$ are non-separable and contained in an infinite $\omega$-limit set;
	\item $\langle x,y\rangle $ is an IN-pair;
	\item $\langle x,y\rangle $ is an IT-pair.
\end{enumerate}	
\end{thm}
 
\subsection{Graph maps}
Recall that a topological space $X$ is a \emph{continuum} if it is a compact, connected metric
space. An \emph{arc} is a continuum homeomorphic to the closed interval $[0,1]$. 
A \emph{(topological) graph} is a continuum which is a union of finitely many arcs, any
two of which are either disjoint or intersect in at most one common endpoint.
Let $G$ be a graph and $f\colon G\to G$ be a continuous map.
Then $(G,f)$ forms a dynamical system. In this case, we say that $f$ is a graph map and usually use the map $f$ to represent the dynamical system $(G,f)$.

In what follows, we will need the following two results on graph maps. To some extent, these are weaker versions of stated earlier results for interval maps, which is the price paid
for a greater generality of space.

\begin{thm}[{\cite[Theorem 3]{RS14}}]
Let $f\colon G\to G$ be a graph map.
If $f$ has a scrambled pair, then it is Li-Yorke chaotic.	
\end{thm} 

\begin{thm}[{\cite[Theorem 1.5]{LOYZ17}}] \label{thm:graph-tame-null-eq}
Let $f\colon G\to G$ be a graph map. Then the following are equivalent:
\begin{enumerate}
	\item $f$ does not have scramble pairs;
	\item $f$ is tame;
	\item $f$ is null.
\end{enumerate}
\end{thm}

\section{The structure of $\omega$-limit sets of graph maps}
In this section,
we take a close look at the structure of $\omega$-limit sets for graph maps,
which were studied by Blokh in the series of papers \cite{Blokh82,Blokh,Blokh90a,Blokh90b}
and \cite{Blokh90c}.
Here we will follow the notations from~\cite{RS14}. 

Let $f\colon G\to G$ be a graph map. A subgraph $K$ of $G$
is called \emph{periodic} if
there exists a positive integer $k$ such that $K,f(K),\dotsc,f^{k-1}(K)$
are pairwise disjoint and $f^k(K)=K$.
In such a case, $k$ is called the \emph{period} of $K$, and
$K$ is a \emph{periodic subgraph with period $k$}.
The set $\orb_f(K)=\bigcup_{i=0}^{k-1} f^i(K)$ is called a \emph{cycle}
of graphs with period $k$.

Let $x\in G$. If the $\omega$-limit set $\omega_f(x)$ is infinite, we define 
\begin{equation*}
\mathcal{C}(x):=\{K\subset G\colon K\text{ is a cycle of graphs and }\omega_f(x)\subset K\}.
\end{equation*}
Note that $\bigcap_{n=1}^\infty f^n(G)$ is always a $1$-periodic cycle of graphs contained in $\mathcal{C}(x)$, hence always $\mathcal{C}(x)\neq \emptyset$.
Denote
$$
\mathcal{C}_P(x)=\sup\{ n\in \bbn : \text{ there is a subgraph }K \text{with period }n\text{ and }\orb_f(K)\in \mathcal{C}(x)\}
$$
and observe that $\mathcal{C}_P(x)\in \bbn\cup \{+\infty\}$.

\begin{defn}
	Let $f\colon G\to G$ be a graph map and $x\in G$. If the $\omega$-limit set $\omega_f(x)$ is infinite and $\mathcal{C}_P(x)=+\infty$,
we say that the $\omega$-limit set $\omega_f(x)$  a \emph{solenoid}.
\end{defn}

\begin{lem}[{\cite[Lemma~10]{RS14}}] \label{lem:solenoid}
	Let $f\colon G\to G$ be a graph map. 
	If $\omega_f(x)$ is a solenoid, then
	there exists a sequence of cycles of graphs $(X_n)_{n=1}^\infty$ with periods $(k_n)_{n=1}^\infty$ such that
	\begin{enumerate}
		\item $(k_n)_{n=1}^\infty$ is strictly increasing, $k_1\geq 1$ and $k_{n+1}$ is a multiple of $k_n$ for all $n\geq 1$;
		\item for all $n\geq 1$, $X_{n+1}\subset X_n$ and every connected component of $X_n$
		contains the same number (equal to $k_{n+1}/k_n\geq 2$) of connected components of $X_{n+1}$;
		\item $\omega_f(x)\subset \bigcap_{n\geq 1}X_n$ and $\omega_f(x)$ does not contain periodic points.
	\end{enumerate}
\end{lem}

\begin{lem}\label{lem:solenoid2}
	Let $f\colon G\to G$ be a graph map. 
Assume that $\omega_f(x)$ is a solenoid. Let $(X_n)_{n=1}^\infty$ and $(k_n)_{n=1}^\infty$ be provided by Lemma \ref{lem:solenoid}.
\begin{enumerate}
	\item If $Y$ is a cycle of graphs with $Y\cap \omega_f(x)\neq\emptyset$, then $\omega_f(x)\subset Y$
	and there exists $i\in\bbn$ such that $X_i\subset Y$.
	\item If  $(Y_n)_{n=1}^\infty$ 
	is a sequence of cycles of graphs with periods $(m_n)_{n=1}^\infty$ 
	such that $Y_{n+1}\subset Y_n$, $\omega_f(x)\subset Y_n$ for all $n\in\bbn$ and $\lim_{n\to\infty}m_n=\infty$, then 
	$\bigcap_{n=1}^\infty X_n= \bigcap_{n=1}^\infty Y_n$.
\end{enumerate}
\end{lem}
\begin{proof}
(1) Pick a point $y\in Y\cap \omega_f(x)$.
By Lemma~\ref{lem:solenoid}(3), the orbit of $y$ is infinite.
As $Y$ has only finitely many boundary points,
there exists $i\in\bbn$ such that $f^i(y)$ is in the interior of $Y$.
It is clear that $f^i(y)\in\omega_f(x)$ and then there exists $j\in\bbn$ such that $f^j(x)$ is in the interior of $Y$. This implies that $\omega_f(x)\subset Y$ as $Y$ is $f$-invariant.

Let $P=\bigcap_{n=1}^\infty X_n$. 
Note that by  Lemma~\ref{lem:solenoid}(2) the set $P$ has uncountably many degenerate connected components and clearly each of them is in $\omega_f(x)$.
As $Y$ has a finite number of connected components and $\omega_f(x)\subset Y$,
there exists a connected component $Y_0$ of $Y$ such that 
$Y_0$ contains uncountably many degenerate connected components of $P$. In particular, the interior of $Y$ contains a degenerate connected component of $Y$.
By the construction of $Y$, there exists $m\in\bbn$
such that there exists a connected component of $X_m$ which is contained in the interior of $Y$. Then $X_m\subset Y$.

(2) By (1), $\bigcap_{n=1}^\infty X_n\subset  \bigcap_{n=1}^\infty Y_n$. 
Fix $n\in\bbn$. By the symmetry,  switching roles of  $(Y_n)_{n=1}^\infty$ and $(X_n)_{n=1}^\infty$ in (1) we obtain that there exists $i\in\bbn$
such that $Y_i\subset X_n$. Therefore, 
$\bigcap_{n=1}^\infty Y_n \subset \bigcap_{n=1}^\infty X_n$.
\end{proof}

\begin{defn}\label{defn:P(x)}
	Let $f\colon G\to G$ be a graph map. 
	Let $\omega_f(x)$ be a solenoid and $(X_n)_{n=1}^\infty$ be provided by Lemma \ref{lem:solenoid}. We denote
	\begin{equation}
	P(x)=\bigcap_{n=1}^{\infty}X_n. \label{def:Px}
	\end{equation}
\end{defn}

\begin{rem}\label{rem:P(x)}
By Lemma~\ref{lem:solenoid2} we clearly have that
\[
P(x)=\bigcap_{K\in \mathcal{C}(x)}K.
\]
This means that definition of $P(x)$ depends only  on $x$ but not on the choice of cycle of graphs $(X_n)_{n=1}^\infty$ provided in Lemma \ref{lem:solenoid}.
\end{rem}

\begin{lem}\label{lem:solenoid cup}
Let $f\colon G\to G$ be a graph map and $x,y\in G$. 
If $\omega_f(x)$ is a solenoid and $\omega_f(x)\cap\omega_f(y)\neq\emptyset$, 
	then $\omega_f(x)\cup \omega_f(y)$ is also a solenoid. 	
\end{lem}
\begin{proof}
	Let $P=\omega_f(x)\cup \omega_f(y)$ and $Q=\omega_f(x)\cap \omega_f(y)$.
	Let $(X_n)_{n=1}^\infty$ and $(k_n)_{n=1}^\infty$ be provided by Lemma \ref{lem:solenoid} for the solenoid $\omega_f(x)$.
	As $G$ has only finitely many branching points,
	for sufficiently large $n$, there exists a connected component $J_n$ of $X_n$
	such that $J_n$ does not contain branch points of $G$.
	Then $f^{k_n}|_{J_n}$ is an interval map.  
	Note that $Q$ is an $f$-invariant closed set and  $Q\subset \omega_f(x)$ has no periodic points. 
	Then $Q\cap J_n$ is infinite. 
	As $J_n$ has only two boundary points, there are infinite many points of $Q$ in the interior of $J_n$. 
	Therefore, there exist $i,j\in\bbn$ such that $f^i(x)$ and $f^j(y)$ are in the interior of $J_n$. 
	Moreover,   $\omega_{f^{k_n}}(f^{i+m}(x))\subset f^m(J_n)$
	and $\omega_{f^{k_n}}(f^{j+m}(y))\subset f^m(J_n)$  for all $0\leq m<k_n$. 
	It is clear that 
	$$
	\omega_{f^{k_n}}(f^i(x))\cap \omega_{f^{k_n}}(f^j(y))=\omega_f(x)\cap \omega_f(y)\cap J_n=Q\cap J_n
	$$ 
	is an infinite set. 
	Then any limit point of $\omega_{f^{k_n}}(f^i(x))\cap \omega_{f^{k_n}}(f^j(y))$ is a
	limit point from the left or from the right of both $\omega_{f^{k_n}}(f^i(x))$ and $\omega_{f^{k_n}}(f^j(y))$. By Theorem \ref{thm:interval limit set}, $\omega_{f^{k_n}}(f^i(x))\cup \omega_{f^{k_n}}(f^j(y))$ is also an $\omega$-limit set of $f^{k_n}|_{J_n}$. So 
	
	$$
	\bigcup_{m=0}^{k_n-1}f^m(\omega_{f^{k_n}}(f^i(x))\cup \omega_{f^{k_n}}(f^j(y)))
	=\omega_f(x)\cup \omega_f(y)
	=P
	$$
	is an $\omega$-limit set of $f$. 
	For all $n>0$, $P\subset X_n$, then $P$ is a solenoid.  
\end{proof}

Note that the set of all $\omega$-limit sets of a graph map is closed under Hausdorff metric \cite{MS2007}, and so each $\omega$-limit set is contained in a maximal one by Zorn Lemma. 

We will also use the following result about the intersection of a sequence of nested cycles of graphs.
\begin{lem}\label{lem:seq-cycle-graphs}
	Let $f\colon G\to G$ be a graph map. 
Let $(X_n)_{n=1}^\infty$ be a sequence of cycles of graphs  with periods $(k_n)_{n=1}^\infty$ such that $X_{n+1}\subset X_n$ and $\lim_{n\to\infty}k_n=\infty$.
Then there exists a unique maximal $\omega$-limit set $\omega_f(z)$ in $\bigcap_{n=1}^\infty X_n$, and $\omega_f(x)\subset \omega_f(z)$ for all $x\in G$ with $\omega_f(x)\cap \bigcap_{n=1}^\infty X_n\neq\emptyset$.	
\end{lem}
\begin{proof}
Let $P=\bigcap_{n=1}^\infty X_n$. It is clear that $P$ dose not contain periodic points. 
Let $x\in G$ with $\omega_f(x)\cap P\neq\emptyset$. Then there exists $y\in \omega_f(x)\cap P$. As $\omega_f(x)\cap P$ is an $f$-invariant closed set, $\omega_f(y)\subset \omega_f(x)\cap P\subset \bigcap_{n=1}^\infty X_n$. This implies that $\omega_f(y)$ is a solenoid. Let $\omega_f(z)$ be the maximal $\omega$-limit set with $\omega_f(z)\supset \omega_f(y)$. It is clear that $\omega_f(z)$ is a solenoid too. By Lemma \ref{lem:solenoid cup} $\omega_f(z)\cup \omega_f(x)$ is also an $\omega$-limit set, therefore $\omega_f(x)\subset \omega_f(z)$. Therefore, $\omega_f(z)$ contains all $\omega$-limit sets in $P$, it is the unique maximal $\omega$-limit set in $P$. 
\end{proof}

\begin{defn}
	Let $X$ be a finite union of subgraphs of $G$ such that $f(X)\subset X$. We define
	\begin{equation*}
	E(X,f)=\{y\in X\colon \text{ for any neighborhood } U \text{ of }y \text{ in } X,\;
	\overline{\orb_f(U)}=X\}.
	\end{equation*}
\end{defn}

The following result first appeared in \cite{Blokh90a}, see also in \cite{RS14} for details. It involves a special kind of semi-conjugacy.

\begin{defn}\label{df:ac}
	Let $f\colon X\to X$ and $g\colon Y\to Y$ be two continuous maps and $E\subseteq X$ be a closed invariant set. A continuous surjection $\varphi\colon X\to Y$ is an \emph{almost conjugacy} between $f|_E$ and $g$ if $\phi \circ f=g \circ \phi$ and
	\begin{enumerate}
		\item $\phi(E)=Y,$
		\item $\forall y\in Y, \phi^{-1}(y)$ is connected,
		\item $\forall y\in Y, \phi^{-1}(y)\cap E=\partial \phi^{-1}(y)$, where $\partial A$ denotes the boundary of $A$.
	\end{enumerate}
\end{defn}

\begin{lem}\label{lem:non-solenoid}
	Let $f\colon G\to G$ be a graph map. 
	If an infinite $\omega$-limit set $\omega_f(x)$ is not a solenoid,
	then there exists a cycle of graphs $X\in\mathcal{C}(x)$ such that
	\begin{enumerate}
		\item for any $Y\in \mathcal{C}(x)$, $X\subset Y$;
		\item the period of $X$ is maximal among the periods of all cycles in $\mathcal{C}(x)$.
	\end{enumerate}
	If we put $E=E(X,f)$ then:
	\begin{enumerate}
		\item $\omega_f(x)\subset E\subset X$;
		\item $E$ is a perfect set and $f|_E$ is transitive;
		\item there exits a transitive map $g\colon Y\to Y$,
		where $Y$ is a finite union of graphs, and a semi-conjugacy $\varphi\colon X\to Y$
		between $f|_X$ and $g$ which almost conjugates $f|_E$ and $g$.
	\end{enumerate}
\end{lem}

\begin{defn}
	Let $f\colon G\to G$ be a graph map. 
	Assume that $\omega_f(x)$ is an infinite  set but not a solenoid.
	Let $X$ be the minimal (in the sense of inclusion) cycle of graphs contained $\omega_f(x)$ and denote $E=E(X,f)$.
	We say that $E$ is a \emph{basic set} if $X$ contains a periodic point, and \emph{circumferential set} otherwise.
\end{defn}

The following theorem first appeared in \cite{Blokh}, here the statement  is a combination of Theorems 20 and 22 in \cite{RS14}.

\begin{thm} \label{thm:basic-set-and-rotation}
	Let $f\colon G\to G$ be a graph map.
	\begin{enumerate}
		\item If $f$ admits a basic set, then the topological entropy of $f$ is positive.
		\item If $f$ is transitive with no periodic points then
		it is conjugate to an irrational rotation of the circle.
	\end{enumerate}
\end{thm}

\begin{lem}\label{lem:solenoid-per}
Let $f\colon G\to G$ be a graph map and $x\in G$.
\begin{enumerate}
	\item If $\omega_f(x)$ is a solenoid, then $\omega_f(x)\cap \overline{\per(f)}$ is uncountable, where $\per(f)$ is the collection of all periodic points of $f$.
	\item If $\omega_f(x)$ is not a solenoid and  $E(X,f)$ with $X$ provided by Lemma~\ref{lem:non-solenoid} is a circumferential,
	then $\omega_f(x)\cap \overline{\per(f)}$ is finite.
\end{enumerate}	
\end{lem}

\begin{proof}
(1) Let $(X_n)_{n=1}^\infty$ and $(k_n)_{n=1}^\infty$ be provided by Lemma \ref{lem:solenoid}. 
As $G$ has only finite many branching points,
for sufficiently large $n$, there exists a connected component $J_n$ of $X_n$
such that $J_n$ does not contain branch points of $G$.
Then $f^{k_n}|_{J_n}$ is an interval map and must contain some periodic point for $f^{k_n}$. 
This implies that every connected component  of $X_n$ contains some periodic point for $f$. By the construction of $P(x)$, 
every connected component  of $P(x)$ contains a point which is a limit point of periodic points. 
So $\omega_f(x)\cap \overline{\per(f)}$ is uncountable.

(2) By the definition of circumferential set,
the cycle of graphs $X$ does not contain periodic points.
Then $X\cap \overline{\per(f)}$ is contained in the boundary of $X$ which is finite. 
Note that $\omega_f(x) \subset X$, so $\omega_f(x)\cap \overline{\per(f)}$ is finite as well.
\end{proof}

\begin{lem}\label{lem:solenoid-f-fm}
	Let $f\colon G\to G$ be a graph map and $x\in G$.
	Then the following are equivalent:
	\begin{enumerate}
		\item $\omega_f(x)$ is a solenoid for $f$;
		\item for every $m\in\bbn$,
		$\omega_{f^m}(x)$ is a solenoid for $f^m$;
		\item there exists some $m\in\bbn$ such that
		$\omega_{f^m}(x)$ is a solenoid for $f^m$.
	\end{enumerate}
\end{lem}
\begin{proof}
(1)$\Rightarrow$(2)
Let $(X_n)_{n=1}^\infty$ and $(k_n)_{n=1}^\infty$ be provided for $\omega_f(x)$ by Lemma \ref{lem:solenoid}.	
Fix $m\in\bbn$. For every $n\in\bbn$, 
there exists a connected component $J_n$ of $X_n$ intersecting $\omega_{f^m}(x)$. Then $J_n$ is a periodic subgraph for $f^m$ with
period $s_n\in [\frac{k_n}{m}, k_n]$, and $\omega_{f^m}(x)
\subset \orb_{f^m}(J_n)$.
It is clear that $\lim_{n\to\infty} s_n=\infty$ as $\lim_{n\to\infty} k_n=\infty$.
So $\omega_{f^m}(x)$ is a solenoid for $f^m$.

(2) $\Rightarrow$(3) is clear.

(3)$\Rightarrow$(1) By Lemma \ref{lem:solenoid-per},
$\omega_{f^m}(x)\cap \overline{\per(f)}$ is uncountable and so
$\omega_f(x)\cap \overline{\per(f)}$ is also uncountable.
Assume that $\omega_f(x)$ is not a solenoid and let $E(X,f)$ be defined for $X$ provided by Lemma~\ref{lem:non-solenoid}. 
If $E(X,f)$ is a circumferential set, then by Lemma~\ref{lem:solenoid-per}, $\omega_f(x)\cap \overline{\per(f)}$ is finite, 
which is a contradiction.

If $E(X,f)$ is a  basic set, it means that $X$ contains a periodic point. Let the dynamical system $(Y,g)$ and semi-conjugacy $\varphi$ be provided by Lemma~\ref{lem:non-solenoid}. By the properties of the semi-conjugacy we have that $\varphi\circ f^n=g^n \circ \varphi$, and then $(Y,g)$ contains a periodic point. By the definition of the almost conjugates $\varphi(E)=Y$ and $\varphi^{-1}(y)\cap E$ is finite for any $y\in Y$, so $E$ contains a periodic point too. Let $E=\omega_f(z)$. Then $\omega_{f^m}(f^i(z))$ contains a periodic point for all $0\leq i<m$, and there exists $0\leq j<m$ such that $\omega_{f^m}(f^j(z))\cap \omega_{f^m}(x)\neq \emptyset$. Based on the given conditions  $\omega_{f^m}(x)$ is a solenoid for $f^m$, by the Lemma \ref{lem:solenoid cup}, then $\omega_{f^m}(f^j(z))$ is a solenoid for $f^m$. By Lemma \ref{lem:solenoid} 
there are no periodic points of $f^m$ in $\omega_{f^m}(f^j(z))$ , so $\omega_{f^m}(f^j(z))$ does not contain periodic points of $f$, 
which is also a contraction.
Thus $\omega_f(x)$ is a solenoid.
\end{proof}

\section{Non-separable points and NS-pairs}\label{sec:NSpairs}
The aim of this section is to prove Theorem \ref{thm:chaos-NS-pair}.
At first, we introduce the concept of separable points for graph maps, which is a natural extension of definitions existing for interval maps. Let us recall at this point, that periodic subgraphs are closed by definition.
\begin{defn}
Let $f\colon G\to G$ be a graph map and let $x$, $y$ be two distinct points in $G$. The points $x,y$ are called \emph{separable} (for $f$) if there exist two disjoint periodic subgraphs $I$ and $J$ such that $x\in  I$ and $y\in J$. Otherwise we say that they are \emph{non-separable}.

It is clear that if $x,y$ are separable for $f$ then for every $k\in\bbn$, $x,y$ are also separable for $f^k$. 
In other words, if there exists $k\in\bbn$ such that 
$x,y$ are non-separable for $f^k$, then $x,y$ are also non-separable for $f$.
\end{defn}

First we have the following results on non-separable pairs.

\begin{lem}\label{lem:non-seq-eq}
Let $f\colon G\to G$ be a graph map, $z\in G$ with $\omega_f(z)$ being a solenoid and $x,y\in \omega_f(z)$.
Then the following are equivalent:
\begin{enumerate}
	\item $x$ and $y$ are non-separable;
	\item there exists a sequence $(J_n)_{n=1}^{\infty}$ of periodic subgraphs with periods $(m_n)_{n=1}^{\infty}$
	such that $J_n\supset J_{n+1}$, $\lim_{n\to\infty} m_n=\infty$
	and $x,y\in J_n$ for all $n\in\bbn$;
	\item  $x,y$ are in the same connected 
	component of $P(z)$, where $P(z)$ is defined as in \eqref{def:Px}.
\end{enumerate}
\end{lem}
\begin{proof}
Let $(X_n)_{n=1}^\infty$ and  $(k_n)_{n=1}^\infty$ be provided as in 
Lemma~\ref{lem:solenoid} for $\omega_f(z)$. 
By the definition of cycle of graphs,
each $X_n$ is the orbit of a periodic subgraph which is denoted by $I_n$. We may also requite that $I_{n+1}\subset I_n$.

(1)$\Rightarrow$(2)
Since $x,y\in X_n$ are non-separable,
there exists $0\leq i\leq k_n-1$ such that $x,y\in f^i(I_n)$.
Let $J_n= f^i(I_n)$ and $m_n=k_n$.
Clearly $J_{n+1}\subset J_n$ and so (2) is satisfied.

(2)$\Rightarrow$(3) By the definition of $P(z)$, one has
\[
P(z)=\bigcap_{n=1}^\infty X_n=\bigcap_{n=1}^\infty \bigcup_{i=0}^{m_n-1}f^i(J_n).
\]
So $x,y$ are in the same connected 
component of $P(z)$.

(3) $\Rightarrow$(1) 
Let $K$ be a periodic subgraph with period $p$ such that $x\in K$. Then $\bigcup_{i=0}^{p-1} f^i(K)$ is a cycle of graphs, and $x,y\in \omega_f(z)\subset \bigcup_{i=0}^{p-1} f^i(K)$ by Lemma \ref{lem:solenoid2}.
By the remark \ref{rem:P(x)} $P(z)\subset \bigcup_{i=0}^{p-1} f^i(K)$. Assume that $x,y$ are in the same connected 
component of $P(z)$, then $x,y$ are in the same connected 
component of $\bigcup_{i=0}^{p-1} f^i(K)$, so $y\in K$. Therefore $x$ and $y$ are non-separable. 
\end{proof}

\begin{lem}\label{lem:non-seq-eq2}
Let $f\colon G\to G$ be a graph map and let $x, y \in \omega(f)$ be two distinct points.
If there exists a sequence $(J_n)_{n=1}^{\infty}$ of periodic subgraphs with periods $(m_n)_{n=1}^{\infty}$
such that $J_n\supset J_{n+1}$, $\lim_{n\to\infty} m_n=\infty$
and $x,y\in J_n$ for all $n\in\bbn$, then 
\begin{enumerate}
	\item there exists a point $z\in G$ such that $x,y\in\omega_f(z)$;
	\item $\omega_f(x)$, $\omega_f(y)$ and $\omega_f(z)$ are solenoids;
	\item $x$ and $y$ are non-separable. 
\end{enumerate}
\end{lem}
\begin{proof}
(1) For every $n\in\bbn$, let $X_n=\orb_f(J_n)$.
By Lemma \ref{lem:seq-cycle-graphs}, there exists a unique maximal $\omega$-limit set $\omega_f(z)$ in $\bigcap_{n=1}^\infty X_n$ 
and $x,y\in\omega_f(z)$.

(2) It is clear that $\omega_f(x),\omega_f(x),\omega_f(z)\subset X_n$
for all $n\in\bbn$. Then $\omega_f(x)$, $\omega_f(y)$ and $\omega_f(z)$ are solenoids.

(3) It follows from Lemma~\ref{lem:non-seq-eq}, (1) and (2).
\end{proof}

For an interval map with zero topological entropy, we are interested in non-separable points contained in an infinite $\omega$-limit set. Note that in this case the $\omega$-limit set is a solenoid.
Based on Lemmas \ref{lem:non-seq-eq} and \ref{lem:non-seq-eq2}, we may introduce the following
definition of NS-pair for graph map. The advantage is that it does not explicitly involve the notion of solenoid in the definition.

\begin{defn}
Let $f\colon G\to G$ be a graph map.
A pair $\langle x,y\rangle\in G\times G$ with $x\neq y$ is an \emph{NS-pair} if $x,y\in \omega(f)$ and 
there exists 
a sequence $(J_n)_{n=1}^{\infty}$ of periodic subgraphs with periods $(m_n)_{n=1}^{\infty}$
such that $J_n\supset J_{n+1}$, $\lim_{n\to\infty} m_n=\infty$
and $x,y\in J_n$ for all $n\in\bbn$.
\end{defn}

\begin{rem}
	By Lemma \ref{lem:non-seq-eq2} we obtain that a pair $\langle x,y\rangle$ is an NS-pair 
	if and only if $x$, $y$ are non-separable and they are contained in some solenoid. 
\end{rem}

\begin{prop}\label{hipow:nonsep}
Let $f\colon G\to G$ be a graph map 
and $\langle x,y\rangle\in G\times G$ with $x\neq y$.
Then the following statements are equivalent:
\begin{enumerate}
	\item $\langle x,y\rangle$ is an NS-pair for $f$;\label{NS:1}
	\item $\langle x,y\rangle$ is an NS-pair for $f^k$ for some $k\in\bbn$;\label{NS:2}
	\item $\langle x,y\rangle$ is an NS-pair for all $f^k$ with $k\in\bbn$.\label{NS:3}
\end{enumerate}	
\end{prop}
\begin{proof}
$\eqref{NS:3}\Rightarrow \eqref{NS:2}$ It is obvious. 
	
$\eqref{NS:2}\Rightarrow\eqref{NS:1}$ 
By Lemma~\ref{lem:non-seq-eq2}, there exists $z\in G$ such that 
$\omega_{f^k}(z)$ is a solenoid for $f^k$ and $x,y\in \omega_{f^k}(z)$.
By Lemma~\ref{lem:solenoid-f-fm},
$\omega_f(z)$ is also a solenoid for $f$.
As $x$ and $y$ are non-separable for $f^k$, it is easy to see that $x$ and $y$ are also non-separable for $f$. Now by Lemma \ref{lem:non-seq-eq}, $\langle x,y\rangle$ is an NS-pair for $f$.
	
$\eqref{NS:1}\Rightarrow \eqref{NS:3}$ 
Assume that  $\langle x,y\rangle$ is an NS-pair for $f$ and fix any integer $k>1$.
Then $x,y\in \omega(f)=\omega(f^k)$.
It is clear that the sequence $(J_n)_{n=1}^{\infty}$ of periodic subgraphs with periods $(m_n)_{n=1}^{\infty}$ 
provided by definition of NS-pair is also periodic under $f^k$, 
and new period $s_n$ of $J_n$ is an integer $s_n\in [m_n/k, m_n]$, in particular $\lim_{n\to\infty} s_n=\infty$.
So $\langle x,y\rangle$ is an NS-pair for $f^k$.
\end{proof}

\begin{lem}\label{lem:NS-pair-preimage}
	Let $f\colon G\to G$ be a graph map and $\langle x,y\rangle$ be an NS-pair. Let periodic subgraphs $(J_n)_{n=1}^{\infty}$ 
	and $(m_n)_{n=1}^{\infty}$ be provided by the definition of the NS-pair $\langle x,y\rangle$. 
	Then for every $k\in\bbn$ and $0\leq i\leq m_k-1$, there exist $x^*\in f^{-(m_k-i)}(x)\cap f^i(J_k)$ 
	and $y^*\in f^{-(m_k-i)}(y)\cap f^i(J_k)$  such that 
	$\langle x^*,y^*\rangle$ is an NS-pair for $f^{m_k}|_{f^i(J_k)}$.
\end{lem}

\begin{proof}
Fix $k\in\bbn$ and $0\leq i\leq m_k-1$.
By Lemma~\ref{lem:non-seq-eq2}, there exists $z\in G$ such that 
$\omega_{f}(z)$ is a solenoid and $x,y\in \omega_{f}(z)$.
Since $\omega_f(z)$ is strongly invariant, there are $x^*,y^*\in \omega_f(z)$ such that $f^{m_k-i}(x^*)=x$ and $f^{m_k-i}(y^*)=y$. 
Then $x^*,y^*\in f^{m_n-(m_k-i)}(J_n)$ for all $n\geq k$.
According to the properties of $(J_n)_{n=1}^{\infty}$, 
we have $(f^{m_n-(m_k-i)}(J_n))_{n=k}^{\infty}$ is a sequence of periodic subgraphs with periods $(s_n)_{n=k}^{\infty}$ for $f^{m_k}|_{f^i(J_k)}$
such that $f^{m_n-(m_k-i)}(J_n)\supset f^{m_{n+1}-(m_k-i)}(J_{n+1})$ for all $n\in\bbn$. 
Since sequence $J_n$ is nested, we have $s_n=m_n/m_k$, in particular $\lim_{n\to\infty} s_n=\infty$. 
Then $\langle x^*,y^*\rangle$ is an NS-pair for $f^{m_k}|_{f^i(J_k)}$.
\end{proof}

\begin{proof}[Proof of Theorem~\ref{thm:chaos-NS-pair}]
($\Rightarrow$) Assume that $\langle x,y\rangle $ is a scrambled pair.
By the proof of \cite[Theorem 3]{RS14} one of the sets  $\omega_f(y)$, $\omega_f(x)$ is a solenoid,
so without loss of generality assume that  $\omega_f(x)$ is a solenoid.
Let $(X_n)_{n=1}^\infty$ be provided for $\omega_f(x)$ by Lemma \ref{lem:solenoid}. 
Note that for every $n$ there is $m>n$ such that for any connected component $C$ of $X_n$
there is a connected component $D$ of $X_m$ such that $D\subset \Int C$.
Fix $n\in\bbn$. 
Note that condition $\liminf_{k\to\infty}d(f^k(x),f^k(y))=0$ implies that
there exists 
$N$ such that $f^N(y)\in X_n$ and furthermore $f^N(x)$ and $f^N(y)$
are in the same connected component of $ X_n$.
This implies that
 for every $k\geq N$, $f^k(x)$ and $f^k(y)$ are in the same 
connected component of $X_n$.
Let $\delta=\limsup_{k\to\infty}d(f^k(x),f^k(y))>0$.
There exists $a_n\geq N$ such that $d(f^{a_n}(x),f^{a_n}(y))\geq \frac{\delta}{2}$.
Without loss of generality, assume that the sequence $(a_n)$
is strictly increasing and the following limits exist: $\lim_{n\to\infty}f^{a_n}(x)=x^*$
and $\lim_{n\to\infty}f^{a_n}(x)=y^*$.
Then $x^*,y^*\in \omega(f)$, they are in the same connected component of $P(x)$ and $d(x^*,y^*) \geq \frac{\delta}{2}$.
Indeed $\langle x^*,y^*\rangle $ is an NS-pair.

($\Leftarrow$)
Assume that $\langle x,y\rangle $ is an NS-pair. 
Let $(J_n)_{n=1}^{\infty}$  and  $(m_n)_{n=1}^{\infty}$ be provided by the definition of the NS-pair $\langle x,y\rangle$. 
As $G$ has only finitely many branch points,
there exist $n\in\bbn$ and $0\leq i\leq m_n-1$
such that $f^i(J_n)$ does not contain any branch point of $G$.
By Lemma~\ref{lem:NS-pair-preimage}, there exists an NS-pair $\langle x^*,y^*\rangle $ 
for $f^{m_n}|_{f^i(J_n)}$.
Note that $f^{m_n}|_{f^i(J_n)}$ is an interval map.
By \cite[Theorem 2.2]{S86} (see also \cite{R17}), a zero
entropy interval map is Li-Yorke chaotic if and only if there exists an infinite $\omega$-limit set containing two non-separable points.
Therefore $f^{m_n}|_{f^i(J_n)}$ is Li-Yorke chaotic, which immediately implies that $f$ is Li-Yorke chaotic.
\end{proof}

\section{Independence pairs in graph maps with zero topological entropy}

The aim of this section is to prove Theorem \ref{thm:IN=IT=NS}.
First we will also use the following two results.

\begin{lem}[{\cite[Lemma~5.9]{LOYZ17}}] \label{lem:IN-omega}
	Let $f\colon G\to G$ be a graph map.
	If $\langle x,y\rangle$ is an IN-pair, then $x,y\in\omega(f)$.
\end{lem}

\begin{prop}[{\cite[Proposition 2.10]{LOZ}}]\label{prop:LOZ}
	If $f\colon G\to G$ is a graph map with zero topological entropy,
	then there exists a continuous map $g$ acting on a graph $Y$ without scrambled
	pairs and a factor map $\pi\colon (G,f)\to (Y,g)$ such that the pair $\langle p,q\rangle$ is asymptotic whenever
	$p,q\in \pi^{-1}(y)$ for some $y\in Y$. 	
\end{prop} 

\begin{lem}\label{lem:IN-solenoid}
Let $f\colon G\to G$ be a graph map with zero topological entropy. If $\langle x,y\rangle$ is a non-diagonal IN-pair, then $\langle x,y\rangle$ is asymptotic and $\omega_f(x)$ is a solenoid. 
\end{lem}

\begin{proof}
Let $\pi\colon (G,f)\to (Y,g)$ be the factor map in Proposition~\ref{prop:LOZ}.
By Lemma~\ref{INIT-conjugate} $\langle \pi(x),\pi(y)\rangle$ is an IN-pair for $(Y,g)$.

As $(Y,g)$ has no scramble pairs, results of \cite[Theorem~1.5]{LOYZ17} imply that $\pi(x)=\pi(y)$, therefore
$\langle x,y\rangle$ is asymptotic. 
At this point, we have to look deeper into the proof of Proposition~\ref{prop:LOZ}. The non-degenerate fibers of $\pi$
are obtained by collapsing non-degenerate connected components of solenoids (there are at most countably many such sets). In particular it means that 
$\omega_f(x)=\omega_f(y)$ 
is a solenoid, completing the proof.
\end{proof}

Now we are ready to prove Theorem \ref{thm:IN=IT=NS}.
\begin{proof}[Proof of Theorem \ref{thm:IN=IT=NS}]
First note, that if $\langle x,y\rangle$ is an IT-pair then it is also an IN-pair by the definition, so $\eqref{Graph:IT}\Rightarrow\eqref{Graph:IN}$.

Next, assume that $\langle x,y\rangle$ is an IN-pair. By Lemma~\ref{lem:IN-omega} we have $x,y\in \omega(f)$.
By Lemma~\ref{lem:IN-solenoid} the $\omega$-limit set $\omega_f(x)$ is a solenoid and therefore so is the maximal $\omega$-limit set $\omega_f(z)$ containing $x$, so directly from the definition we
obtain a sequence of periodic subgraphs $J_n$
such that $x\in J_n$ for every $n\in\bbn$. But since $x,y$ are asymptotic, we have $\omega_f(y)\subset \omega_f(z)$
and therefore $y\in \omega_f(z)\subset J_n$.
This shows that $\langle x,y\rangle$ is an NS-pair, that is $\eqref{Graph:IN}\Rightarrow\eqref{Graph:NS}$.

For the remaining implication $\eqref{Graph:NS}\Rightarrow\eqref{Graph:IT}$, assume that $\langle x,y\rangle$ is an NS-pair. By the definition we obtain periodic subgraphs $(J_n)_{n=1}^{\infty}$ 
and sequence of associated periods $(m_n)_{n=1}^{\infty}$, such that $x,y\in J_n$ for every $n\in\bbn$.
Choose sufficiently large integer $k\in\bbn$ such that for some $0\leq i<m_k-1$ the map $f^{m_k}|_{f^i(J_k)}$ is an interval map.
By Lemma \ref{lem:NS-pair-preimage}, there exist $x^*,y^* \in f^i(J_k)$ such that $\langle x^*,y^*\rangle$ is an NS-pair for $f^{m_k}|_{f^i(J_k)}$, $f^{m_k-i}(x^*)=x$ and  $f^{m_k-i}(y^*)=y$.
By Theorem~\ref{thm:interval-IN-IT},
$\langle x^*,y^*\rangle$ is an IT-pair for $f^{m_k}|_{f^i(J_k)}$.
By Lemma~\ref{lem:IT-IN-fk}, $\langle x^*,y^*\rangle$ is an IT-pair for $f$, and $\langle x,y\rangle$ is also an IT-pair by Lemma~\ref{lem:IT-IN-fk}, as $f^{m_k-i}\times f^{m_k-i}\langle x^*,y^* \rangle=\langle x,y\rangle$.
\end{proof}

\section*{Acknowledgments}

J. Li and X. Liang were supported in part by NSF of China (11771264, 
11871188) and NSF of Guangdong (2018B030306024).
P. Oprocha was supported in part by National Science Center, Poland, grant no. 2019/35/B/ST1/02239.
The authors would like to thank the anonymous referees for their help comments.

\end{document}